\documentclass[a4paper,11pt,reqno]{amsart}

\usepackage[T1]{fontenc}
\usepackage[utf8]{inputenc}
\usepackage{lmodern}
\usepackage[english]{babel}

\usepackage[%
left=2.5cm,       
right=2.5cm,      
top=3.5cm,        
bottom=3.5cm,     
heightrounded,    
bindingoffset=0mm 
]{geometry}

\allowdisplaybreaks

\usepackage{amsrefs}
\usepackage{amsmath}
\usepackage{amssymb}
\usepackage{mathtools}
\usepackage{mathrsfs}
\usepackage[colorlinks=true,linkcolor=blue,urlcolor=blue,citecolor=blue]{hyperref}

\numberwithin{equation}{section}

\usepackage{hyperref}

\usepackage{tikz}
\usepackage{graphicx}
\usepackage{xcolor}
\usepackage[font=small]{caption}

\usepackage{bookmark}
\usepackage{esint}

\allowdisplaybreaks

\theoremstyle{plain}
\newtheorem{theorem}{Theorem}[section]

\newtheorem{proposition}[theorem]{Proposition}

\newtheorem*{mtheorem}{Main Theorem}
\theoremstyle{definition}

\newtheorem{remark}[theorem]{Remark}
\newtheorem{open.problem}[theorem]{Open Problem}

\newcommand{\N}{\mathbb{N}}
\newcommand{\R}{\mathbb{R}}

\title[Local Regularity of very weak $s$-harmonic functions]{Local Regularity of very weak $s$-harmonic functions via fractional difference quotients}
\author[A.\ Carbotti]{Alessandro Carbotti}
\address{Dipartimento di Matematica
	e Fisica ``E. De Giorgi'', Universit\`a del Salento,
	Via Per Arnesano, 73100 Lecce, Italy.}
\email{alessandro.carbotti@unisalento.it}
\author[S.\ Cito]{Simone Cito}
\address{Dipartimento di Matematica
	e Fisica ``E. De Giorgi'', Universit\`a del Salento,
	Via Per Arnesano, 73100 Lecce, Italy.}
\email{simone.cito@unisalento.it}
\author[D. A. \ La Manna]{Domenico Angelo La Manna}
\address{Dipartimento di Matematica e Applicazioni ``R. Caccioppoli'', Universit\`a degli studi di Napoli
	``Federico II'', Complesso di Monte Sant'Angelo, Via Cintia, 80126 Naples, Italy.}
\email{domenicoangelo.lamanna@unina.it}
\author[D.\ Pallara]{Diego Pallara}
\address{Dipartimento di Matematica
	e Fisica ``E. De Giorgi'', Universit\`a del Salento, and INFN, Sezione di Lecce,
	Via Per Arnesano, 73100 Lecce, Italy.}
\email{diego.pallara@unisalento.it}

\date{\today}  



\begin{document}
		\begin{abstract}
The aim of this paper is to give a new proof that any very weak $s$-harmonic function $u$ in the unit ball $B$ is smooth. As a first step, we improve the local summability properties of $u$. Then, we exploit a suitable version of the difference quotient method tailored to get rid of the singularity of the integral kernel and gain Sobolev regularity and local linear estimates of the $H^{s}_{\rm loc}$ norm of $u$. Finally, by applying more standard methods, such as elliptic regularity and Schauder estimates, we reach real analyticity of $u$. Up to the authors' knowledge, the difference quotient techniques are new.
	\end{abstract}
	
	\maketitle
	
	\section{Introduction}

This paper comes from our attempt to generalize the by now classical difference quotient method due to L. Nirenberg to nonlocal operators. It has been introduced in \cite{Nirenberg} and is now presented in all the textbooks dealing with regularity properties of solutions of elliptic equations. After the introduction of weak, or even distributional, solutions of partial differential equations, the problem of their regularity has been tackled by various techniques. Probably the first result in this direction has been the proof of regularity of weakly harmonic functions, obtained in the fifties by Hermann Weyl in \cite{weyl} and by Renato Caccioppoli, see \cite[page 122]{Miranda}. Subsequently, much more general operators have been considered and one of the most fruitful and flexible techniques has proved to be that of difference quotients, which - as it is - appears to be strictly depending on the {\em local} character of differential operators. We refer e.g. to  \cite{LLS, zhangbao} and the references therein for recent regularity results relying on Nirenberg method. 

On the other hand, the notion of distributional solution is well established also for equations coming from {\em nonlocal} operators and the question on the regularity of such solutions is in turn quite natural. One of the first examples of nonlocal operator, and probably the simplest one, is the fractional power of the laplacian and solutions of the equation $(-\Delta)^su=0$, $s\in(0,1)$, are called $s$-harmonic. The class of $s$-harmonic functions has been broadly studied in the last years. Though they have many features relating them with harmonic functions, see for instance \cites{fallweth, DGV}, $s$-harmonic functions exhibit a different behaviour in other aspects, due to the nonlocal nature of the fractional Laplacian. Among these facts we mention the local density of $s$-harmonic functions among smooth functions (\cite{DSV17}), a purely nonlocal phenomenon that has some interesting conse\-quences, such as the failure of the classical Harnack inequality and a quantitative version of an inverse mean value formula in the fractional case. See \cites{kassmann, BDV} for more precise statements, \cite{CDV19} for more general density results and \cites{abaval, bucval} for other applications. 

There are several equivalent ways of defining $(-\Delta)^s$, see \cite{kwasnicki}, and the first proof of local regularity of $s$-harmonic distributions has been obtained via pseudodifferential techniques by R. T. Seeley \cite{Seeley1}. See \cite{Seeley2} and \cite{Seeley3} for more general operators. The Dirichlet problem for the fractional Laplacian
	 \begin{equation}
	 \label{eq:dualproblem}
	 \begin{cases}
	 (-\Delta)^s w=f\quad\text{in}\quad\Omega \\
	 w=0\quad\text{in}\quad\Omega^c
	 \end{cases}
	 \end{equation}
	 has been studied in \cites{LPPS, rosotonserra13} providing basic summability estimates according to the summability of the source $f$ and boundary regularity, respectively. 
	 %Improved Sobolev regularity for weak solutions of equations more general than \eqref{eq:dualproblem} is proved in \cite{cozzi}. 
	 We notice that when $\Omega$ is the whole of $\R^N$ more than $H^s(\R^N)$ regularity is available for weak solution of $(-\Delta)^sw=f$, see \cite{KMS}.
	 
It seemed to us that a natural question about the regularity of very weak solutions of nonlocal equations, see \eqref{eq:defveryweaksol} for the precise definition in the case of the fractional laplacian, is if it is possible to extend the difference quotient method to such solutions. As a first check, we have considered the operator $(-\Delta)^s$, where the difficulty of getting rid of the singularity of the kernel in its definition, see \eqref{eq:fractionallaplacian} below, already appears. 
  
Very weak solutions of \eqref{eq:dualproblem} have been treated in \cite{abdperal}, where the authors observe that such solutions, when $\Omega$ is a bounded smooth domain and $f\in L^1(\Omega)$, are actually pointwise solutions, that is they are given in terms of the Green operator applied to the source $f$. See \cite{bucur}, where explicit representation formulae when $\Omega$ is a ball are given. We also mention that in the fractional setting maximal regularity estimates for weak solutions of \eqref{eq:dualproblem} are also available: in \cite{grubb} the author proves that the solution of \eqref{eq:dualproblem} with $f\in L^p(\Omega)$ belongs to the  fractional Sobolev space $W^{2s,p}_{\text{loc}}(\Omega)$, or to the Besov space $B^{2s}_{p,2, \text{loc}}(\Omega)$ according to the values of $s$ and $p$, and this regularity is sharp since does not hold true up to the boundary. See also \cite{BWZ, BWZa}. Anyway, our arguments do not rely on the estimates proved there. 

We consider very weak solutions of the following problem
	 \begin{equation}
	 \label{eq:maineq}
	 \begin{cases}
	 (-\Delta)^s w=0\quad\text{in}\quad B \\
	 w=g\quad\text{in}\quad B^c,
	 \end{cases}
	 \end{equation}
	  where $B$ is the unit ball in $\R^N$ and the outer datum $g$ belongs to the weighted Lebesgue space $L^1_s(\R^N)$ that is defined in \eqref{defL1s}. See formula \eqref{eq:defveryweaksol} for the precise definition. Notice that our assumptions allow to deal with also with $s$-harmonic functions that blow up at the boundary of $B$, see \cites{abatangelo, BBKRSV}. If the external datum is assumed to be bounded, one can also exploit the explicit representation of the solution in terms of the fractional Poisson kernel, see \cites{BogByc, bucur, rofr}. 

Our main result is the following 
	 \begin{mtheorem}
	 	\label{th:maintheorem}
	 	Let $u$ be a very weak solution of \eqref{eq:maineq}. Then
	 	\begin{enumerate}
	 	\item (Sobolev regularity) $u\in H^s_{\rm loc}(B)$ and the estimate 
	 	$$
	 	\left\|u\right\|_{H^{s}(B')}\le c(B')\left\|u\right\|_{L^1_s(\R^N)}
	 	$$
	 	holds for every $B'\Subset B$.
	 	\item (Classical regularity) $u$ is real analytic in $B$ and the estimate
	 	\[
	 	\left\|D^\iota u\right\|_{L^\infty(B_{r_0})}\le c^{|\iota|}\iota!C(R,r_0,N,s)\left(\left\|u\right\|_{L^\infty(B_R)}+\left\|u\right\|_{L^1_s(\R^N)}\right)
	 	\]
	 	holds for any $\iota\in\N_0^N$ and $0<r_0<R<1$.
	 	\end{enumerate}
	 \end{mtheorem}
	 We prove our Main Theorem in several steps. In Theorem \ref{th:summabilityimprovement} we prove that the solution $u$, which is by definition in $L^1(B)$, is actually in $L^2_{\text{loc}}(B)$. This local improvement of summability is done by suitable localisation methods joint with regularity results for the fractional Poisson equation with homogeneous external condition proved in \cites{afly, rosotonserra13}.   The second step is done in Theorem \ref{th:sobolevregularity}, where we prove that the $L^2_{\rm loc}(B)$ solution $u$ belongs to the fractional Sobolev space $H^{2s}_{\text{loc}}(B)$, see Section \ref{sec:notandprel} for the relevant definitions. The main point in the proof of this result consists in showing the $H^{s}_{\text{loc}}(B)$ regularity, as the final step from $H^{s}_{\text{loc}}(B)$ to $H^{2s}_{\text{loc}}(B)$ follows from $L^2$ estimates on the operator $I_s$ which is the {\em carré du champ} of the fractional Laplacian that arises using the relevant fractional Leibniz rule. This kind of estimates, that we also prove for every $p\ge 1$, are different with respect to the one proved in \cite{KePoVe}. We exploit a suitable variant of the classical Nirenberg difference quotient method: we introduce a cut-off function that vanishes near the origin and allows us to get rid of the singularity of the kernel and to obtain the fractional Sobolev regularity $H^s_{\text{loc}}(B)$. Difference quotient methods have been used in \cite{AFV, BrLi, CafSil, cozzi} in a different fashion to improve regularity of solutions to more general nonlocal equations. We point out once more that the core of the paper are the linear estimates and the new techniques introduced to prove Claim 1 in the Main Theorem.
	 
	 In the third step we prove that for a cut-off function $\eta$ the function $\eta^2u$ solves the equation $(-\Delta)^sw=f$ in the whole space, with $f\in L^2(\R^N)$, and as a consequence $u$ belongs to $H^{2s}_{\rm loc}(B)$.
	 
	 The fourth step relies on fractional De Giorgi estimates proved in \cite{BrLiSc} that allow us to gain local boundedness of $u$ in $B_1$ and also local H\"older regularity. The usefulness of those estimates with respect to the previous literature (see e.g.~\cite{kassm07, Kass}) leans on the fact that the H\"older continuity exponent is quantified. Namely the authors prove that $u\in C^{0,\gamma}_{\rm loc}(B_1)$ for every $\gamma\in\left(0,\min\{2s,1\}\right)$.
	 
	 In the fifth step we use again that $\eta^2u$ solves the equation $(-\Delta)^sw=f$ in $\R^N$ but with $f\in C^{0,\gamma}(\R^N)$ and as a consequence $u$ belongs to $C^{\gamma+2s}_{\rm loc}(B)$ and is $s$-harmonic in the classical sense in $B$. To do this we also exploit suitable H\"older continuity properties of the operator $I_s$.
	 
	 In the last step we use the explicit Poisson representation formula to give a pointwise expression for $u$ in a small ball. By well known estimates on the derivatives of the Poisson kernel we conclude our proof by proving the real analyticity of $u$.
	 
Finally, we notice that all our results are stated and proved using the unit ball as reference domain in order to avoid technical issues and to focus on the core of our strategy
%But, since we do not make use of any explicit representation formulae for solutions of the fractional Laplace equation in the ball,
 though the same results also hold true for every bounded and sufficiently smooth domain.

\noindent
\paragraph*{\bf Acknowledgements}  
The authors are grateful to Proff. Gerd Grubb, Xavier Ros-Oton and Enrico Valdinoci for their interest in the paper and their comments in several fruitful discussions. 

The authors are member of GNAMPA of the Istituto Nazionale di Alta Matematica (INdAM). 
A.C. and S.C., D.A.L., D.P. respectively acknowledge the support of the INdAM - GNAMPA 2023 Projects ``Problemi variazionali per funzionali e operatori non-locali'', ``Disuguaglianze isoperimetriche e spettrali'', ``Equazioni differenziali stocastiche e operatori di Kolmogorov in dimensione infinita''. A.C., S.C. and D.A.L. acknowledge the support of the INdAM - GNAMPA 2024 Project ``Ottimizzazione e disuguaglianze funzionali per problemi geometrico-spettrali locali e nonlocali''.
A.C., S.C. and D.P. have been also partially supported by the PRIN 2022 project 20223L2NWK. D.A.L. has been also partially supported by the PRIN 2022 project 2022E9CF89.

\noindent
{\bf Data Availability Statement} Data sharing not applicable to this article as no datasets were generated or analysed during the current study.
	
	\section{Notations and Preliminary results}
	\label{sec:notandprel}
In the whole paper we always assume that $N\ge 2$.

The space $L^1_s(\R^N)$ is the weighted Lebesgue space defined as
\begin{equation}\label{defL1s}
L^1_s(\R^N):=\left\{u\in\mathcal{M}(\R^N);\quad\left\|u\right\|_{L^1_s(\R^N)}<\infty\right\},
\end{equation}
where $\mathcal{M}(\R^N)$ denotes the space of Lebesgue measurable functions on $\R^N$ and 
$$
\left\|u\right\|_{L^1_s(\R^N)}:=\int_{\R^N}\frac{|u(x)|}{1+|x|^{N+2s}}dx.
$$ 
It is very easy to check that  $L^p(\R^N)\subsetneq L^1_s(\R^N)\subsetneq L^1_{\text{loc}}(\R^N)$ for every $p\ge 1$. 

The space $L^1_s(\R^N)$ is a natural setting for very weak $s$-harmonic functions. Indeed, it encodes local integrability and a growth condition at infinity. This is equivalent to requiring that the nonlocal tail of $u$
$$
\text{Tail}(u;x_0,R):=R^{2s}\int_{B^c_R(x_0)}\frac{|u(x)|}{|x-x_0|^{N+2s}}dx
$$
is finite for every $x_0\in\R^N$ and $R>0$.

For $u\in C_{\text{loc}}^{2s+\gamma}(B)\cap L^1_s(\R^N)$, $\gamma\in(0,1)$, the $s$-Laplacian $(-\Delta)^s u$ is pointwise defined for every $x\in B$ and the following representation formula holds
\begin{equation}
\label{eq:fractionallaplacian}
(-\Delta)^su(x)=C_{N,s}\int_{\R^N}\frac{u(x)-u(y)}{|x-y|^{N+2s}}dy=\frac{C_{N,s}}{2}\int_{\R^N}\frac{2u(x)-u(x+y)-u(x-y)}{|y|^{N+2s}}dy,
\end{equation}
where $C_{N,s}:=\frac{s2^{2s}\Gamma\left(\frac N2+s\right)}{\pi^{N/2}\Gamma(1-s)}$ and $\Gamma$ denotes the Euler Gamma function. This choice of the normalisation constant makes the fractional Laplacian a Fourier multiplier with symbol $|\cdot|^{2s}$ whenever for $u\in L^1(\R^N)$ the Fourier transform $\mathcal{F}$ is defined as $\mathcal{F}u(\xi)=\int_{\R^N}u(x)e^{-2\pi ix\cdot\xi}dx$. 

Notice that if $u$ only belongs to $L^1_s(\R^N)$, formula \eqref{eq:fractionallaplacian} still holds true by taking the integrals in the Cauchy principal value sense.

For $s\in(0,1)$, $1\le p<\infty$ and $\Omega\subseteq\R^N$ we define the fractional Sobolev space $W^{s,p}(\Omega)$ as
$$
W^{s,p}(\Omega)=\left\{u\in L^p(\Omega):\quad [u]_{W^{s,p}(\Omega)}<\infty\right\},
$$
where
$$
[u]_{W^{s,p}(\Omega)}:=\left(\int_\Omega dx\int_\Omega\frac{|u(x)-u(y)|^p}{|x-y|^{N+sp}}dy\right)^{1/p},
$$
endowed with the norm $\left\|\cdot\right\|_{W^{s,p}(\Omega)}:=\left(\left\|\cdot\right\|^p_{L^p(\Omega)}+[\cdot]^p_{W^{s,p}(\Omega)}\right)^{1/p}$. 

When $p=\infty$ any $f\in W^{s,\infty}(\Omega)$ has a representative $\tilde{f}\in C^s(\overline{\Omega})$.

As usual, when $p=2$ we use the notation $H^s(\Omega)$ to indicate the Hilbert space $W^{s,2}(\Omega)$. See \cite{DPV} for a gentle introduction to the fractional Sobolev spaces.

Let us also define higher order fractional Sobolev spaces, confining to the non-integer case: for $\sigma\in(1,\infty)$, $\sigma=k+s$, $k\in\N$, $s\in(0,1)$ and $1\le p<\infty$, the fractional Sobolev space $W^{\sigma,p}(\Omega)$ is defined as follows:
$$
W^{\sigma,p}(\Omega)=\left\{u\in W^{k,p}(\Omega):\quad D^\alpha u\in W^{s,p}(\Omega),\quad\forall\ \alpha\in\N_0^N,\ |\alpha|\leq k\right\}.
$$
Set $\mathcal{Q}_A:=\left(A\times A\right)\cup\left(A\times A^c\right)\cup\left(A^c\times A\right)$ for every open set $A$. We define
\begin{equation}\label{defH2sLoc}
\mathbb{H}^s(B)=\left\{u\in L^2(B):\quad [u]_{\mathbb{H}^s(B)}<\infty\right\},
\end{equation}
where
$$
[u]_{\mathbb{H}^s(B)}:=\Bigl(\iint_{\mathcal{Q}_B}\frac{|u(x)-u(y)|^2}{|x-y|^{N+2s}}dxdy\Bigr)^{1/2}.
$$
We say that $u\in \mathbb{H}^s(B)$ is a weak solution for \eqref{eq:maineq} if for every $\varphi\in \mathbb{H}^s_0(B)=H^s_0(B)=\overline{C^\infty_c(B)}^{\left\|\cdot\right\|_{H^s(B)}}$ it holds that
$$
\begin{cases*}
\displaystyle\int_{\R^N}dx\int_{\R^N}\frac{(u(x)-u(y))(\varphi(x)-\varphi(y))}{|x-y|^{N+2s}}dy=0 \\
u=g\quad\text{in}\quad B^c.
\end{cases*}
$$
Notice that for $u\in\mathbb{H}^s(B)$ and $\varphi\in\mathbb{H}_0^s(B)$ the definition is well posed. Indeed, let $A\Subset B$, $A\supset \text{supp}\, \varphi$
\begin{align*}
\left|\int_{\R^N}dx\int_{\R^N}
\frac{(u(x)-u(y))(\varphi(x)-\varphi(y))}{|x-y|^{N+2s}}dy\right|&\leq\iint_{\mathcal{Q}_A}
\frac{|u(x)-u(y)||\varphi(x)-\varphi(y)|}{|x-y|^{N+2s}}dxdy
\\
&\leq [u]_{\mathbb{H}^s(A)}[\varphi]_{\mathbb{H}^s(A)}.
\end{align*}

We notice that if $g\in C(\R^N)\cap L^\infty(\R^N)$ and $u\in\mathbb{H}^s(B)\cap L^\infty(\R^N)$ is a weak solution of \eqref{eq:maineq} then $u$ is also a solution in the viscosity sense for \eqref{eq:maineq}, as proved in \cite[Theorem 1]{SerVal} also in the inhomogeneous case for continuous sources.

For $s\in(0,1)$ we also introduce the space 
$$
L^\infty_s(\R^N):=\Bigl\{u\in L^\infty(\R^N);\quad\sup_{x\in\R^N}(1+|x|^{N+2s})|u(x)|<\infty\Bigr\},
$$
equipped with the norm
$$
\left\|u\right\|_{L^\infty_{s}(\R^N)}:=\sup_{x\in\R^N}(1+|x|^{N+2s})|u(x)|.
$$
We say that $u\in L_s^1(\R^N)$ is a very weak solution of \eqref{eq:maineq} if, for every $\varphi$ compactly supported in $B$ such that $(-\Delta)^s\varphi\in L^\infty_{s}(\R^N)$, it holds
\begin{equation}
\label{eq:defveryweaksol}
\begin{cases*}
\displaystyle\int_{\R^N} u(-\Delta)^s\varphi dx=0 
\\
u=g\quad\text{in}\quad B^c.
\end{cases*}
\end{equation}
Notice that the chosen class of test functions is not empty. Indeed, let $\varphi\in C_c^{2s+\gamma}(B)$ for some $\gamma>0$. We have
\begin{align*}
\Biggl|\int_{\R^N}u(-\Delta)^s&\varphi dx\Biggr|\le
\int_{B_2}|u(-\Delta)^s\varphi| dx
+\int_{B_2^c}|u(-\Delta)^s\varphi| dx 
\\
&\le\|u\|_{L^1(B_2)}\|\varphi\|_{C^{2s+\gamma}(B)}
+\int_{B^c_2}\frac{|u(x)|}{1+|x|^{N+2s}}dx
\int_{B}\frac{1+|x|^{N+2s}}{|x-y|^{N+2s}}|\varphi(y)|dy 
\\
&\le\|u\|_{L^1(B_2)}\|\varphi\|_{C^{2s+\gamma}(B)}
+\int_{B^c_2}\frac{|u(x)|}{1+|x|^{N+2s}}dx
\int_{B}\left(1+\frac{1+|y|}{|x-y|}\right)^{N+2s}|\varphi(y)|dy 
\\
&\le \|u\|_{L^1(B_2)}\|\varphi\|_{C^{2s+\gamma}(B)}
+3^{N+2s}~\|u\|_{L^1_s(\R^N)}\|\varphi\|_{L^1(B)}
\\
&\leq C_{N,s} \|\varphi\|_{C^{2s+\gamma}(B)}\|u\|_{L^1_s(\R^N)},
\end{align*}
and then $\varphi$ is a test function.

We notice that a weak solution is a very weak solution. Indeed, using the symmetry of the double integral in $x$ and $y$
$$
\frac 12\int_{\R^N}\int_{\R^N}\frac{(u(x)-u(y))(\varphi(x)-\varphi(y))}{|x-y|^{N+2s}}dxdy=\int_{\R^N}u(x)dx\int_{\R^N}\frac{\varphi(x)-\varphi(y)}{|x-y|^{N+2s}}dy
$$
for every $\varphi\in C^{\infty}_c(B)$.

Now, we recall some useful results. From now on, for $r\ge1$ we denote with $r':=\frac{r}{r-1}$ the H\"{o}lder conjugate of $r$, and for $\tau>0$ such that $r\tau<N$ we denote with $r^*_\tau:=\frac{Nr}{N-r\tau}$ the Sobolev conjugate of $r$ with respect to $\tau$. First, we state the Sobolev embedding theorem. 

\begin{proposition}{\cite[Theorem 7.63]{adams}}
	\label{prop:sobolevembedding}
	Let $\Omega\subseteq\R^N$ an open and smooth set and let $k,h\ge 0$, $p\ge 1$. If $u\in W^{k,p}(\Omega)$, $k>h$ and $N>(k-h)p$ then the following continuity estimate holds
	$$
	\left\|u\right\|_{W^{h,q}(\Omega)}\le C\left\|u\right\|_{W^{k,p}(\Omega)}
	$$
	\begin{itemize}
		\item for every $1\le q\le \frac{Np}{N-(k-h)p}\quad\text{if}\,\,\Omega\,\,\text{has finite measure}$
		\item for every $p\le q\le \frac{Np}{N-(k-h)p}\quad\text{if}\,\,\Omega\,\,\text{has infinite measure}$.
	\end{itemize}
\end{proposition}

We notice that if $k=h\in\mathbb{N}_0$ and if $\Omega$ has finite measure then the statement of Proposition \ref{prop:sobolevembedding} still holds true, but if $k=h\in(0,\infty)\setminus\mathbb{N}$ then the embedding may fail in general even if $\Omega$ is a ball, see \cite{MS}.

The following results give regularity properties of the weak solutions of \eqref{eq:dualproblem} under suitable assumptions on $f$. 

\begin{theorem}{\cite[Corollary 1.7]{afly}}
\label{th:afly}
Let $N\ge 2$, $\Omega\subset\R^N$ a bounded $C^2$ domain, $s\in(0,1)$ and let $u$ be the unique solution of
	\begin{equation*}
	\begin{cases}
	(-\Delta)^s w=f\quad\text{in}\quad\Omega \\
	w=0\quad\text{in}\quad\Omega^c
	\end{cases}
	\end{equation*}
	with $f\in L^m(\Omega)$.
\begin{itemize}
	\item[(i)] If $1\le m<\frac Ns$, then for all $1<p<m_s^*$ there exists $C>0$ such that
	$$
	\left\|u\right\|_{W^{s,p}(\R^N)}\le C\left\|f\right\|_{L^m(\Omega)}.
	$$
	\item[(ii)] If $m>\frac Ns$, then for all $1<p<\infty$ there exists $C>0$ such that
	$$
	\left\|u\right\|_{W^{s,p}(\R^N)}\le C\left\|f\right\|_{L^m(\Omega)}.
	$$
\end{itemize}	
\end{theorem}

\begin{theorem}{\cite[Proposition 1.4]{rosotonserra13}}
	\label{th:rosotonserra}
	Let $s\in(0,1)$, $N>2s$, $\Omega\subset\R^N$ a bounded $C^{1,1}$ domain, $f\in C(\overline{\Omega})$ and let $u$ be the weak solution of
	\begin{equation*}
	\begin{cases}
	(-\Delta)^s w=f\quad\text{in}\quad\Omega \\
	w=0\quad\text{in}\quad\Omega^c.
	\end{cases}
	\end{equation*}
	\begin{itemize}
		\item[(i)] For each $1\le r<\left(\frac{N}{2s}\right)'$ there exists $C=C(n,r,s,|\Omega|)>0$ such that
		$$
		\left\|u\right\|_{L^r(\Omega)}\le C\left\|f\right\|_{L^1(\Omega)}.
		$$
		\item[(ii)] Let $1<p<\frac{N}{2s}$ and $p_{2s}^*=\tfrac{Np}{N-2sp}$. Then there exists $C=C(n,s,p)>0$ such that for any $1\le q\le p^*_{2s}$
		$$
		\left\|u\right\|_{L^q(\Omega)}\le C\left\|f\right\|_{L^p(\Omega)}.
		$$
		\item[(iii)] Let $\frac{N}{2s}<p<\infty$. Then there exists $C=C(n,s,p,\Omega)>0$ such that
		$$
		\left\|u\right\|_{C^\beta(\R^N)}\le C\left\|f\right\|_{L^p(\Omega)},
		$$
		where $\beta:=\min\left\{s,2s-\frac Np \right\}$.
	\end{itemize}
\end{theorem}

\section{Improvement of summability}

Now we are ready to state and prove the following

\begin{theorem}
	\label{th:summabilityimprovement}
	Let $u\in L^1_s(\R^N)$ a very weak solution of \eqref{eq:maineq}. Then $u\in L^2_{{\rm loc}}(B)$.
\end{theorem}

\begin{proof}
To ease the presentation, we divide the proof in three steps.
\\
{\it Step one: first summability improvement}. 
In this first step we prove 
\begin{equation} \label{eq:claim1}
u \in L^r_{\text{loc}}(B)\quad\text{ for all }\quad r<\frac{N}{N-s}. 
\end{equation}
	Let $p>\frac{N}{s}$ and $\psi\in C^\infty(B) \cap  C(\overline{B})$. Now, let $v$ be the unique solution of the Dirichlet problem
	\begin{equation}
	\label{eq:testproblem}
	\begin{cases}
	(-\Delta)^s w=\psi\quad\text{in}\quad B_{1-\delta} \\
	w=0\quad\text{in}\quad B_{1-\delta}^c.
	\end{cases}
	\end{equation}
for some $\delta>0$ sufficiently small to be conveniently chosen in the sequel. Now let $\eta\in C^\infty_c(B)$ be such that $\eta=1$ in $B_{1-4\delta}$, $\eta=0$ in $B^c_{1-2\delta}$ and $|\nabla\eta|\le \frac 1\delta$. Notice that by Theorem \ref{th:rosotonserra} we have that $v \in C^s(\R^N)$, and this easily implies that $(-\Delta)^s(\eta^2 v)\in L_s^\infty(\R^N)$ and ${\rm supp}\, \eta^2 v \subset B_{1-2\delta}$.
Indeed, for $|x|>2$ it holds
$$
|(-\Delta)^s(\eta^2 v)(x)|\le\left|\int_{B}\frac{\eta^2(y)v(y)}{|x-y|^{N+2s}}dy\right|\le \frac{3^{N+2s}\left\|v\right\|_{L^1(B)}}{1+|x|^{N+2s}},
$$
while the boundedness of $(-\Delta)^s(\eta^2 v)$ in $B_2$ is an immediate consequence of Theorem \ref{th:rosotonserra}.

Therefore, we can use $\varphi=\eta^2v$ as a test function in the definition of very weak solution \eqref{eq:defveryweaksol}.	
	Then
	\begin{equation}
	0=\int_{\R^N}\eta^2 u\psi dx+2\int_{\R^N}uv\eta(-\Delta)^s\eta dx-\int_{\R^N}u\eta I_s(\eta,v)dx-\int_{\R^N}uI_s(\eta,\eta v)dx
	\end{equation}
	where for any $f_1,f_2$ measurable we have set
\begin{equation}\label{defIs}
	I_s(f_1,f_2)(x):=C_{N,s}\int_{\R^N}\frac{(f_1(x)-f_1(y))(f_2(x)-f_2(y))}{|x-y|^{N+2s}}dy,
\end{equation}
and the constant $C_{N,s}$ is that one in the definition of $(-\Delta)^s$.

	Then
	\begin{equation}
	\label{eq:firststepsummability}
	\begin{split}
	\left|\int_{\R^N}\eta^2 u\psi dx\right|&\le\int_{\R^N}|uI_s(\eta,\eta v)|dx+2\int_{\R^N}|uv\eta(-\Delta)^s\eta|dx+\int_{\R^N}|u\eta I_s(\eta,v)|dx
	\\
	&=:A_1+A_2+A_3.
	\end{split}
	\end{equation}
	We start by estimating the term $A_1$. Then
	\begin{equation}
	\begin{split}
	\int_{\R^N}|uI_s(\eta,\eta v)|dx&\le \left\|u\right\|_{L_s^1(\R^N)}\left\|I_s(\eta,\eta v)\right\|_{L^\infty_{s}(\R^N)}
	\le\, C(\delta)\left\|u\right\|_{L_s^1(\R^N)}\left\|\eta v\right\|_{C^{s}(\R^N)} \\
	&\le C(\delta)\left\|u\right\|_{L_s^1(\R^N)}\left\|v\right\|_{C^s(\R^N)}\le C(\delta)\left\|u\right\|_{L_s^1(\R^N)}\left\|\psi\right\|_{L^p(B)},
	\end{split}
	\end{equation}
	where the last inequality exploits item (iii) in Theorem \ref{th:rosotonserra}, which holds true since $p>\frac{N}{s}$.
	For the second inequality we notice that for any $x\in\R^N$ and $y\in B_{1-2\delta}$ we have 
	\begin{equation}
	\label{eq:estdistboundary}
	\begin{split}
	(1+|x|^{N+2s})|\eta(x)-\eta(y)|\le&\, C(\delta) \chi_{B_{1-\delta}}(x) \left\|\nabla\eta\right\|_{L^\infty(\R^N)}|x-y| 
	\\
	&+2^{N+2s-1}\chi_{B_{1-\delta}^c} (x)(1+|x-y|^{N+2s}+|y|^{N+2s}) \\
	\le&\, C(\delta) \chi_{B_{1-\delta}}(x) \left\|\nabla\eta\right\|_{L^\infty(\R^N)}|x-y| \\
	&+C(N,s,\delta)\chi_{B_{1-\delta}^c} (x)|x-y|^{N+2s}
	\end{split} 
	\end{equation}
	while for any $x\in\R^N$ and $y\in B^c_{1-2\delta}$
	$$
	(1+|x|^{N+2s})|\eta(x)-\eta(y)|\le C(\delta)\chi_{B_{1-\delta}}(x).
	$$
Therefore, for any $x\in\R^N$
	\begin{align}
	\label{eq:trepuntootto}
	(1+|x|^{N+2s})&\left|I_s(\eta,\eta v)(x)\right|\le (1+|x|^{N+2s})\int_{B_{1-2\delta}}\frac{|\eta(x)-\eta(y)||\eta(x)v(x)-\eta(y)v(y)|}{|x-y|^{N+2s}}dy 
	\\ \nonumber
	&+
	(1+|x|^{N+2s})\int_{B^c_{1-2\delta}}\frac{|\eta(x)-\eta(y)||\eta(x)v(x)-\eta(y)v(y)|}{|x-y|^{N+2s}}dy 
	\\ \nonumber
	\le&\,(1+|x|^{N+2s})\int_{B_{1-2\delta}}\frac{|\eta(x)-\eta(y)||\eta(x)v(x)-\eta(y)v(y)|}{|x-y|^{N+2s}}dy 
	\\ \nonumber
	&+
	(1+|x|^{N+2s})|v(x)|\int_{B^c_{1-2\delta}}\frac{(\eta(x)-\eta(y))^2}{|x-y|^{N+2s}}dy 
	\\ \nonumber
    \le&\,\chi_{B_{1-\delta}}(x)\left\|\eta v\right\|_{C^s(\R^N)}\int_{B_{1-2\delta}}\frac{dy}{|x-y|^{N+s-1}} 
    \\ \nonumber
    &+\chi_{B_{1-2\delta}}(x)C'(\delta)|v(x)|\left(\int_{B^c_{\delta}(x)}\frac{dy}{|x-y|^{N+2s}}+\int_{B_{2-3\delta}(x)}\frac{dy}{|x-y|^{N+2s-2}}\right) 
    \\ \nonumber
    &+C(N,s,\delta)\left\|\eta v\right\|_{L^\infty(\R^N)}\chi_{B^c_{1-\delta}}(x)
    \\ \nonumber
    \leq& C(N,s,\delta)\left\|v\right\|_{C^{s}(\R^N)}\le C(N,s,\delta)\left\|\psi\right\|_{L^p(B)},
	\end{align}
	where in the third inequality we have used \eqref{eq:estdistboundary}, the equality  $$|\eta(x)-\eta(y)||\eta(x)v(x)-\eta(y)v(y)|=|v(x)||\eta(x)-\eta(y)|^2$$ that holds for any $x\in B_{1-\delta}$ and $y\in B^c_{1-2\delta}$ and also that the integral term in the fourth line of \eqref{eq:trepuntootto} is nonzero if and only if $x\in B_{1-2\delta}$ and it can be split in the sum of two integrals over $B^c_{1-\delta}$ and $B_{1-\delta}\setminus B_{1-2\delta}$.
	
	For $A_2$ we have 
	\begin{equation}
	\begin{split}
	\int_{\R^N}|uv\eta(-\Delta)^s\eta|dx&\le C\left\|u\right\|_{L^1(B_{1-2\delta})}\left\|v\right\|_{L^\infty(B_{1-2\delta})}\le C\left\|u\right\|_{L^1(B_{1-2\delta})}\left\|\psi\right\|_{L^p(B)},
\end{split}
	\end{equation}
where we used again Theorem \ref{th:rosotonserra}.

To estimate $A_3$ we notice that
$$
\int_{\R^N}|u\eta I_s(\eta,v)|dx=\int_{B_{1-2\delta}}|u\eta I_s(\eta,v)|dx\le\left\|u\right\|_{L^1(B_{1-2\delta})}\left\|I_s(\eta,v)\right\|_{L^\infty(B_{1-2\delta})}
$$
and for almost any $x\in B_{1-2\delta}$, if we split $A_3$ into the sum of the integrals over $B_{1-\delta}$ and $B_{1-\delta}^c$ we have 
	$$
	|I_s(\eta,v)(x)|\le C\left\|v\right\|_{C^s(\R^N)}\left(\int_{B_{2-3\delta}(x)}\frac{dy}{|x-y|^{N+s-1}}+\int_{B^c_\delta(x)}\frac{dy}{|x-y|^{N+2s}}\right)\le C\left\|\psi\right\|_{L^p(B)},
	$$
	where we have used again Theorem \ref{th:rosotonserra}(iii), that holds true for $p>\frac{N}{s}$ with $\beta=s$.
	Now we have 
	\begin{equation}\label{eq:stepone}
	\left|\int_{\R^N}u\eta^2\psi dx\right|\le C\left\|u\right\|_{L_s^1(\R^N)}\left(\left\|v\right\|_{W^{s,p}(\R^N)}+\left\|v\right\|_{C^s(\R^N)}\right)\le C\left\|u\right\|_{L_s^1(\R^N)}\left\|\psi\right\|_{L^p(B)}
	\end{equation}
for all $\psi \in C^\infty (B)\cap C(\overline{B})$, with $p>\frac Ns$. By the density of $C^\infty (B)\cap C(\overline{B})$ in $L^p(B)$ we have that \eqref{eq:stepone} holds for all $\psi\in L^p(B)$, which implies that $\eta^2 u \in L^{p'}(B)$, hence $\eta^2 u \in L^r(B)$ for all $r<\frac{N}{N-s}$. The arbitrariness of $\delta$ gives the claim.
	\\
	{\it Step two: higher summability.} Our next goal is to show
\begin{equation}
u \in L^r_{\text{loc}}(B)\quad\text{for}\quad r \in\left(\frac{N}{N-s},\frac{N}{N-2s}\right)
\end{equation}
In order to improve the summability of the solution $u$ we still use a duality argument, but in a bit different way.
Let $\psi \in C^\infty(B)\cap C(\overline{B})$, take $m \in\left(\frac{N}{(1+\alpha)s},\frac{N}{s}\right)$ for some $\alpha\in(0,1)$ and let $v,\eta$ as before, where now $\delta$ is the double of the previous one. Since we know that $u \in L_{\text{loc}}^{r}(B)$ for all $r<\frac{N}{N-s}$,
let $p' \in (\frac{N}{N-\alpha s},\frac{N}{N-s}) $.
It is very easy to check that the function $v\eta^2$ is  admissible as test function in definition \eqref{eq:defveryweaksol}. We estimate again the three terms appearing in \eqref{eq:firststepsummability}, but this time we can use the higher summability of $u$ proved in Step one to estimate 
	\begin{equation}
	\label{eq:stimaII2it}
	\begin{split}
	A_2&\le C\left\|u\right\|_{L^{p'}(B_{1-2\delta})}\left\|v\right\|_{L^{p}(B_{1-2\delta})}\le C\left\|u\right\|_{L^{p'}(B_{1-2\delta})}\left\|v\right\|_{W^{s,q}(B_{1-2\delta})}
	\\
	&\le C\left\|u\right\|_{L^{p'}(B_{1-2\delta})}\left\|\psi\right\|_{L^{m}(B)} ,
\end{split}	
	\end{equation}
	where the second inequality exploits Sobolev embedding Theorem for $q\ge\frac{Np}{N+sp}\ge\frac{N}{2s}$ and the third inequality exploits Theorem \ref{th:afly} and it holds for all $q\in(1,m^*_s)$ (recall that $m^*_s=\frac{Nm}{N-ms}$) whenever $1\le m<\frac Ns$.

Now we estimate again the term $A_1$ in \eqref{eq:firststepsummability}
\begin{equation} \label{eq:step2A1}
\begin{split}
A_1\le&\int_{B_{1-\delta}}|u(x)|dx\int_{B^c_{1-2\delta}}\frac{|\eta(x)-\eta(y)||\eta(x)v(x)-\eta(y)v(y)|}{|x-y|^{N+2s}}dy \\
&+\int_{B_{1-\delta}}|u(x)|dx\int_{B_{1-2\delta}}\frac{|\eta(x)-\eta(y)||\eta(x)v(x)-\eta(y)v(y)|}{|x-y|^{N+2s}}dy \\
&+\int_{B^c_{1-\delta}}|u(x)|dx\int_{\R^N}\frac{|\eta(x)-\eta(y)||\eta(x)v(x)-\eta(y)v(y)|}{|x-y|^{N+2s}}dy \\
=:&\,B_1+B_2+B_3.
\end{split}
\end{equation}
To bound $B_1$ we first observe that since $\eta(y)=0$ for $y\in B^c_{1-2\delta}$ we have 
\begin{equation}\label{eq:stimaalfa}
\begin{split}
B_1=&\,\int_{B_{1-2\delta}}|u(x)||v(x)|\eta^2(x)dx\int_{B^c_{1-2\delta}}\frac{dy}{|x-y|^{N+2s}},
\\
=&\,\int_{B_{1-2\delta}}|u(x)||v(x)|dx\int_{B^c_{1-2\delta}}\frac{(\eta(x)-\eta(y))^2}{|x-y|^{N+2s}}\:dy
\\
\le&\,\int_{B_{1-2\delta}}|u(x)||v(x)|dx\int_{B^c_{1-\delta}}\frac{dy}{|x-y|^{N+2s}}\\\
&+\int_{B_{1-2\delta}}|u(x)||v(x)|dx\int_{B_{1-\delta}\setminus B_{1-2\delta}}\frac{(\eta(x)-\eta(y))^2}{|x-y|^{N+2s}}\:dy\\
\le&\, C(\delta)\left(\left\|u\right\|_{L^{p'}(B_{1-2\delta})}\left\|v\right\|_{L^{p}(B_{1-2\delta})}+\int_{B_{1-2\delta}}|u(x)||v(x)|dx\int_{B_{1-\delta}}\frac{1}{|x-y|^{N+2s-2}}\:dy\right)
\\
\le&\, C(\delta)\left\|u\right\|_{L^{p'}(B_{1-2\delta})}\left\|\psi\right\|_{L^{m}(B)}
\end{split}
\end{equation}
Analogously for $B_3$ we use that $\eta(x)=0$ for $x \in B^{c}_{1-\delta}$ and that $(1+|x|^{N+2s})\le C_\delta|x-y|^{N+2s}$ for any $x\in B^c_{1-\delta}$ and $y\in B_{1-2\delta}$ to find
\begin{equation}
\label{eq:stimagamma}
\begin{split}
B_3&\le\int_{B^c_{1-\delta}}\frac{|u(x)|}{1+|x|^{N+2s}}dx\int_{B_{1-2\delta}}(1+|x|^{N+2s})\frac{|v(y)|\eta^2(y)}{|x-y|^{N+2s}}dy \\
%&=\int_{B^c_{1-2\delta}}\frac{|u(x)|}{1+|x|^{N+2s}}dx\int_{B_{1-2\delta}}\left(\frac{1+|x|^{N+2s}}{|x-y|^{N+2s}}\right)\eta^2(y)|v(y)|dy \\
&\le C_\delta\left\|u\right\|_{L_s^1(\R^N)}\left\|v\right\|_{L^{1}(B_{1-2\delta})} 
\le  C_\delta\left\|u\right\|_{L_s^1(\R^N)}\left\|\psi\right\|_{L^{m}(B)} 
\end{split}
\end{equation}
and again the Sobolev embedding Theorem holds true for any $q\ge\frac{N}{2s}$. To bound $B_2$ we use the H\"{o}lder inequality in the following way:
\begin{align} \label{eq:step2B2}
B_2&\leq \int_{B_{1-\delta}}|u(x)|dx
\int_{B_{1-2\delta}}\frac{|\eta(x)-\eta(y)|
|\eta(x)v(x)-\eta(y)v(y)|}{|x-y|^{N+2s}}dy 
\\ \nonumber
&\leq \|u\|_{L^{p'}(B_{1-\delta})} 
\biggl(\int_{B_{1-\delta}}\biggl( \int_{B_{1-\delta}} 
\frac{|\eta(x)v(x)-\eta(y)v(y)|}{|x-y|^{N+2s-1}} dy   \biggr)^{p}dx \biggr)^{1/p}
\\ \nonumber
&\leq  \|u\|_{L^{p'}(B_{1-\delta})}
\biggl(\int_{B_{1-\delta}}\int_{B_{1-\delta}} 
\frac{|\eta(x)v(x)-\eta(y)v(y)|^p}{|x-y|^{N+sp}} dy   
\cdot \biggl(\int_{B_{1-\delta}} \frac{dy}{|x-y|^{N-(1-s)p'}}\biggr)^{p-1} dx \biggr)^{1/p}
\\ \nonumber
&\leq C  \|u\|_{L^{p'}(B_{1-\delta})}\|\eta v\|_{W^{s,p}(\R^N)} \leq C \|u\|_{L^{p'}(B_{1-\delta})}\|\psi \|_{L^m(B)},
\end{align}
provided that $p<\frac{Nm}{N-sm}$. Since $p' >\frac{N}{N-\alpha s}$ we get $p <\frac{N}{\alpha s}$, hence 
$\|v \|_{W^{s,p}(B)}\leq \|\psi \|_{L^m(B)}$ if $m > \frac{N}{\left(1+\alpha\right)s}$.
Inequalities \eqref{eq:stimaalfa},\eqref{eq:stimagamma},\eqref{eq:step2B2} give
\begin{equation} \label{eq:step2A2}
|A_2|\leq C(\delta) \|\psi \|_{L^m(B)}\|u\|_{L^{p'}(B_{1-2\delta})}.
\end{equation}
To estimate $A_3$ we proceed as in the previous estimate:
\begin{align*}
A_3=\int_{B_{1-2\delta}}|u\eta I_s(\eta,v)|dx&\le\int_{B_{1-2\delta}}| \eta(x) u(x)|\int_{\R^N}\frac{|\eta(x)-\eta(y)||v(x)-v(y)|}{|x-y|^{N+2s}}dy\, dx .
\end{align*}
We split the integral on $\R^N$ in $B_{1-\delta}$ and $B_{1-\delta}^c$ and we use again the H\"{o}lder inequality and Theorem \ref{th:afly} to infer
\begin{equation} \label{eq:step2A3a}
\begin{split}
\int_{B_{1-2\delta}} &|\eta(x)u(x)|\int_{B_{1-\delta}}\frac{|\eta(x)-\eta(y)||v(x)-v(y)|}{|x-y|^{N+2s}}dy\, dx
\\
\leq& 
\left( \int_B |\eta u|^{p'}dx\right)^{1/p'}
\left(\int_{B_{1-\delta}}\left( \int_{B_{1-\delta}} \frac{|v(x)-v(y)|}{|x-y|^{N+2s-1}} \,dy \right)^{p}\,dx\right)^{1/p}
\\
\leq &
 C \|\eta u\|_{L^{p'}(B)}\biggl(\int_{B_{1-\delta}}dx \int_{B_{1-\delta}} \frac{|v(x)-v(y)|^p}{|x-y|^{N+sp}}\,dy\biggr)^{1/p}
\\
\leq & 
 C \|\eta u\|_{L^{p'}(B)} \|v\|_{W^{s,p}(\R^N)} \leq C\|\eta u\|_{L^{p'}(B)} \|\psi\|_{L^m(B)}.
\end{split}
\end{equation}
Concerning the second integral, since $v(y)=\eta(y)=0$ for $y\in B^c_{1-\delta}$, we have
\begin{equation} \label{eq:step2A3b}
\begin{split}
\int_B|\eta(x)u(x)|dx\int_{B^c_{1-\delta}}&\frac{|\eta(x)-\eta(y)||v(x)-v(y)|}{|x-y|^{N+2s}}dy
\\
&=\int_{B_{1-2\delta}} |\eta(x)u(x)v(x)|dx\int_{B^c_{1-\delta}}\frac{dy}{|x-y|^{N+2s}} 
\\
&\leq
C(\delta) \|\eta u\|_{L^{p'}(B)} \|v\|_{L^p(B)}
\leq C(\delta) \|\eta u\|_{L^{p'}(B)} \|\psi\|_{L^m(B)}
\end{split}
\end{equation}
provided that $m>\frac{N}{(1+\alpha )s}$.
Inequalities \eqref{eq:step2A3a} and \eqref{eq:step2A3b} give $|A_3|\leq C(\delta)\|\eta u\|_{L^{p'}(B)}  \|\psi\|_{L^m(B)}$. Hence, using the latter and inequalities \eqref{eq:step2A1}, \eqref{eq:step2A2} we deduce
\[
\left|\int_{B} u \eta^2 \psi \, dx \right| \leq C(\delta)\|\eta u\|_{L^{p'}(B)}  \|\psi\|_{L^m(B)}
\]
for all $\psi \in C^\infty(B)\cap L^\infty(B)$. By density the inequality holds for $\psi \in L^m(B)$ and hence $\eta^2 u \in L^{m'}(B)$ with 
$m' \in (\frac{N}{N-s},\frac{N}{N-(1+\alpha)s})$. Since this is true for any $\alpha\in(0,1)$ we get our claim.

{\it Step three.} We finally show that $u \in L_{\rm loc}^p(B)$ for  $p<\frac Ns$.
\\
We first prove recursively that
\[
u \in L^r_{\text{loc}}(B)\quad\text{for}\quad r<\frac{N}{N-ks}
\]
for all $k\in \mathbb{N}$ such that $k<\frac Ns$. We notice that we already proved the claim for $k=1,2$.
Hence let us assume that $u\in L^{p'}_{\text{loc}}(B)$ for $p'\in\left(\frac{N}{N-(k-\alpha)s},\frac{N}{N-ks}\right)$ for some $\alpha\in(0,1)$. Fix $\delta>0$ to be chosen again as the double of the one selected in the previous step, let $\psi \in C^\infty (B)\cap L^\infty(B)$, take $m\in\left(\frac{N}{(k+1-\alpha)s},\frac{N}{ks}\right)$, and let $v$ be the unique solution of \eqref{eq:dualproblem}. For a cut-off function $\eta$ supported in $B_{1-2\delta}$ we use $\eta^2 v$ as a test function in \eqref{eq:defveryweaksol} to find again
\[
\left|\int_{B}\eta^2 u \psi dx\right| \leq A_1+A_2+A_3,
\]
with $A_1,A_2,A_3$ defined as in \eqref{eq:firststepsummability}.
As before, to estimate $A_1$ we split it in three terms: 
%\allowdisplaybreaks
\begin{align} \label{eq:step3A1}
A_1\le&\int_{B_{1-\delta}}|u(x)|dx\int_{B^c_{1-2\delta}}\frac{|\eta(x)-\eta(y)||\eta(x)v(x)-\eta(y)v(y)|}{|x-y|^{N+2s}}dy 
\\ \nonumber
&+\int_{B_{1-\delta}}|u(x)|dx\int_{B_{1-2\delta}}\frac{|\eta(x)-\eta(y)||\eta(x)v(x)-\eta(y)v(y)|}{|x-y|^{N+2s}}dy 
\\ \nonumber
&+\int_{B^c_{1-\delta}}|u(x)|dx\int_{\R^N}\frac{|\eta(x)-\eta(y)||\eta(x)v(x)-\eta(y)v(y)|}{|x-y|^{N+2s}}dy 
\\ \nonumber
=:&\,B_1+B_2+B_3
\end{align}
The same argument used to bound $B_1$ and $B_3$ in Step two provides
\[
|B_1| \leq C(\delta) \|u\|_{L^{p'}(B_{1-2\delta})}\|v\|_{W^{s,p}(\R^N)}
\]
and 
\[
|B_3| \leq C(\delta) \|u\|_{L^1_s(\R^N)}\|v\|_{W^{s,p}(\R^N)}
\]
where at this stage $p\in\left(\frac{N}{ks}, \frac{N}{(k-\alpha)s}\right)$.
For $B_2$ we use H\"{o}lder inequality to have
\[
\begin{split}
|B_2|&\le \int_{B_{1-\delta}}|u(x)|dx\int_{B_{1-2\delta}}\frac{|\eta(x)-\eta(y)||\eta(x)v(x)-\eta(y)v(y)|}{|x-y|^{N+2s}}dy 
\\
&\leq \|u\|_{L^{p'}(B_{1-\delta})} 
\biggl(\int_{B_{1-\delta}}\biggl( \int_{B_{1-\delta}} 
\frac{|\eta(x)v(x)-\eta(y)v(y)|}{|x-y|^{N+2s-1}} dy \biggr)^{p}dx \biggr)^{1/p}
\\
&\leq \|u\|_{L^{p'}(B_{1-\delta})}\biggl(\int_{B_{1-\delta}}\int_{B_{1-\delta}} \frac{|\eta(x)v(x)-\eta(y)v(y)|^p}{|x-y|^{N+sp}} dy   
\cdot \biggl(\int_{B_{1-\delta}} \frac{dy}{|x-y|^{N+(s-1)p'}}\biggr)^{p-1} dx \biggr)^{1/p}
\\
&\leq C  \|u\|_{L^{p'}(B_{1-\delta})}\|\eta v\|_{W^{s,p}(\R^N)} 
\end{split}
\]
Since $p>1$ we have
\[
|A_2|\leq  C(\delta)  \|u\|_{L^{p'}(B_{1-\delta})}\|v\|_{W^{s,p}(\R^N)} 
\]
and using Theorem \ref{th:afly} we get
\[
\|v\|_{W^{s,p}(\R^N)} \leq C \|\psi \|_{L^m(B)}
\]
whenever 
\[
p < \frac{Nm}{N-sm} \qquad \text{i.e.,} \qquad m > \frac{Np}{N+sp}.
\]
Since $p< \frac{N}{(k-\alpha)s}$ we get  
\[
|A_2| \leq C(\delta)  \|u\|_{L^{p'}(B_{1-\delta})}\|\psi\|_{L^m(B)}  
\]
if $m>\frac{N}{(k+1-\alpha)s}$.
The estimate for $A_3$ follows from the same argument and gives
\[
|A_3|\leq C(\delta)  \|u\|_{L^{p'}(B_{1-\delta})}\|\psi\|_{L^m(B)}.
\]
Thus we arrive at 
\[
\left|\int_B \eta^2 u \psi\, dx\right|\leq C(\delta) \|u\|_{L^{p'}(B_{1-\delta})}\|\psi\|_{L^m(B)}
\]
for any $m>\frac{N}{(k+1-\alpha)s}$. Using again a duality argument and since the latter is true for all $\alpha >0$ we get
\[
u \in L_{\text{loc}}^{m'}(B) 
\]
for any $m'<\frac{N}{N-(k+1)s}$.
 
Hence we now run this argument $k_0$ times, where $k_0:=\max\{d\in\mathbb{N};\,d\le\overline{k}\}$ and $\overline{k}:=\frac Ns-1$ to find 
$$
u \in L^r_{\text{loc}}(B)
$$
for all $r <\frac{N}{N-k_0s}\le\frac Ns$. Since in particular $2<\frac 2s\le\frac Ns$ by our assumptions in Section \ref{sec:notandprel} the proof is complete. 
\end{proof}

\begin{remark}
In each step $k$ of the proof of Theorem \ref{th:summabilityimprovement} we choose $\delta_k>0$ such that $\delta_k<\frac 14$ and $\delta_k=2\delta_{k-1}<1$, for $k\in\{1,\ldots,k_0\}$ and these conditions imply that in Step one we have to fix $\delta_1:=\delta<\frac{1}{2^{k_0+2}}$.
\end{remark}

\begin{remark}\label{rem:stimeI}
We notice that from estimates \eqref{eq:stimaalfa}, \eqref{eq:stimagamma}, \eqref{eq:step2B2} and \eqref{eq:step2A3a}, we deduce that for any $w\in L^1_s(\R^N)\cap W^{s,p}_{\text{loc}}(B)$ for some $1<p<\infty$, by definition $\eta w\in W^{s,p}(\R^N)$ for any $\eta\in C^\infty_c(B)$ cut-off function and, if ${\rm supp}\,\eta=B_{1-2\delta}\Subset B_{1-\delta}\Subset B$ there exists $C$, that depends only on $\delta, s, N,p$ but independent of $w$, such that
\begin{equation}
\label{eq:stimaIs1}
\left\|I_s(\eta,\eta w)\right\|_{L^p(\R^N)}\le C\left\|\eta w\right\|_{W^{s,p}(\R^N)}.
\end{equation}
Moreover, for any $x\in \R^N$ we have 
%\allowdisplaybreaks
\begin{equation}
\label{eq:stimaIsprel}
\begin{split}
|\eta(x)I_s(\eta,w)(x)|\le&\,\eta(x)\int_{B_{1-\delta}}\frac{|\eta(x)-\eta(y)||w(x)-w(y)|}{|x-y|^{N+2s}}dy 
\\
&+\eta(x)\int_{B^c_{1-\delta}}\frac{|\eta(x)-\eta(y)||w(x)-w(y)|}{|x-y|^{N+2s}}dy 
\\
\le&\, C(\delta)\eta(x)\int_{B_{1-\delta}}\frac{|w(x)-w(y)|}{|x-y|^{N+2s-1}}dy 
\\
&+2\eta(x)\int_{B^c_{1-\delta}}\frac{|w(x)-w(y)|}{|x-y|^{N+2s}}dy 
\\
\le&\, C(\delta)\eta(x)\left(\int_{B_{1-\delta}}\frac{|w(x)-w(y)|^p}{|x-y|^{N+sp}}dy\right)^{1/p}\left(\int_{B_{1-\delta}}\frac{dy}{|x-y|^{N-p'(1-s)}}\right)^{1/p'} 
\\
&+2\eta(x)\left(|w(x)|\int_{B^c_{1-\delta}(x)}\frac{dy}{|x-y|^{N+2s}}+\int_{\R^N}\frac{|w(y)|}{1+|y|^{N+2s}}dy\right).
\end{split}
\end{equation}
From \eqref{eq:stimaIsprel} we deduce that there exists a positive constant $C=C(N,s,\delta,p)$ such that
\begin{equation}
\label{eq:stimaIs2}
\left\|\eta I_s(\eta,w)\right\|_{L^p(\R^N)}\le C(N,s,\delta,p)\left(\left\|w\right\|_{W^{s,p}(B_{1-\delta})}+\left\|w\right\|_{L^1_s(\R^N)}\right).
\end{equation}
The estimates in the cases $p=1$ and $p=\infty$ also hold true with analogous computations. In the case $p=\infty$ we recall that for any $\Omega$ open and smooth set $W^{s,\infty}(\Omega)=C^{0,s}(\overline{\Omega})$, see e.g. \cite[Pag. 59]{DPV}.
\end{remark}

\section{Fractional Sobolev regularity}

In this section we prove the local $H^s$ regularity of very weak solutions, which we know to be in $L^2_{\rm loc}$. In the classical case, Sobolev regularity is usually obtained via the Nirenberg difference quotients method, but in the nonlocal case this method does not work directly because of the presence of a divergent kernel. Therefore, we have devised a different approach, which consists in using another cut-off function (the $\eta_\tau$ below) that eliminates the singularity and makes the relevant integrals convergent. 

\begin{theorem}
	\label{th:sobolevregularity}
	Let $u$ be a very weak solution of \eqref{eq:maineq} in $L^2_{{\rm loc}}(B)$. Then $u\in H^{2s}_{{\rm loc}}(B)$.
\end{theorem}

\begin{proof} Let $u$ be given as in the statement. Let $\tau\in(0,1/2)$ and let $\eta_\tau:[0,+\infty)\to[0,1]$ be the cut-off function defined as
$$
	\eta_\tau(t):=\left\{
	\begin{array}{ll}
	0 &\text{if $0\le t\le\tau/2$,}
	\\
	\frac{2}{\tau}t-1&\text{if $\tau/2\le t\le\tau$,}
	\\
	1&\text{if $\tau\le t$,}
	\end{array}
	\right.
$$
For $\delta\in\left(0,\frac 14\right)$, let us consider another cut-off function $\eta:\R^N\to[0,1]$ such that
	$$
	\eta=1\ \text{in $B_{1-4\delta}$},\quad \eta=0\ \text{in $B^c_{1-2\delta}$},\quad |\nabla\eta|\le\frac 1\delta.
	$$
	For any $x\in\R^N$ let us define the function
	\begin{equation}\label{eq:difquo}
	D^s_{\eta_\tau,\eta}u(x):=\int_{\R^N}\eta_\tau(|x-y|)\frac{\eta(x)u(x)-\eta(y)u(y)}{|x-y|^{N+2s}}\:dy.
	\end{equation}
In order to prove the required regularity, as a test function in \eqref{eq:defveryweaksol} we choose 
	$$
	\varphi(x):=\eta(x)v(x),
	$$
	where $v$ is the solution of the problem
	\begin{equation}\label{eq:auxveryweak}
	\begin{cases}
(-\Delta)^s w=D^s_{\eta_\tau,\eta}u &\text{in $B_{1-\delta}$}
\\
	w=0 &\text{in $B_{1-\delta}^c$}.
	\end{cases}
	\end{equation}
We notice that $\varphi$ is an admissible test function since $(-\Delta)^{s}\varphi \in L^\infty_{s}(\R^N)\subset L^2(\R^N)$. 	
	Since $\eta$ is supported in $B_{1-2\delta}$, we have 
	\begin{equation*}
	\begin{split}
	0&=\int_{\R^N}u(-\Delta)^s\varphi\:dx=\int_{\R^N}u(-\Delta)^s(\eta v)\:dx\\
	&=\int_{B_{1-\delta}}uv(-\Delta)^s\eta\:dx+\int_{B_{1-2\delta}}u \eta(-\Delta)^s v\:dx-\int_{\R^N}uI_s(\eta,v)dx,
	\end{split}
	\end{equation*}
where $I_s$ is defined in \eqref{defIs}. It follows
	\begin{equation}\label{eq:splitestim}
	\begin{split}
	\left|\int_{B_{1-2\delta}}u \eta D^s_{\eta_\tau,\eta}u\:dx\right|&\le\left|\int_{B_{1-\delta}}u v(-\Delta)^s\eta\:dx\right|+\left|\int_{\R^N}uI_s(\eta,v)dx\right|=:|C_1|+|C_2|.
	\end{split}
	\end{equation}
First of all, rewrite the left hand side of \eqref{eq:splitestim} as
	\begin{align}
	\int_{\R^N}u\eta D^s_{\eta_\tau,\eta}u\:dx&=\int_{\R^N}u(x)\eta(x)dx\int_{\R^N}\eta_\tau(|x-y|)\frac{\eta(x)u(x)-\eta(y)u(y)}{|x-y|^{N+2s}}\:dy
\label{eq:estIII}	\\ \nonumber
	&=\frac{1}{2}\int_{\R^N}dx\int_{\R^N}\eta_\tau(|x-y|)\frac{(\eta(x)u(x)-\eta(y)u(y))^2}{|x-y|^{N+2s}}\:dy=:\frac{1}{2} G^s_{\eta_\tau,\eta}(u).
	\end{align}
	Let us estimate the term $C_2$. We write $C_2=C_3+C_4$, where
	$$
	C_3:=\int_{B_{1-\delta}}u(x)dx\int_{\R^N}\frac{(v(x)-v(y))(\eta(x)-\eta(y))}{|x-y|^{N+2s}}\:dy
	$$
	and
	$$
	C_4:=\int_{B^c_{1-\delta}}u(x)dx\int_{\R^N}\frac{(v(x)-v(y))(\eta(x)-\eta(y))}{|x-y|^{N+2s}}\:dy.
	$$
	We have
	\begin{equation}\label{eq:estII.1}
	\begin{split}
	|C_3|&=\biggl|\int_{B_{1-\delta}}u(x)dx\int_{\R^N}\frac{(v(x)-v(y))(\eta(x)-\eta(y))}{|x-y|^{N+2s}}\:dy\biggr|
	\\
	&\le\|u\|_{L^2(B_{1-\delta})}\biggl(\int_{\R^N}\biggl(\int_{\R^N}\frac{(v(x)-v(y))(\eta(x)-\eta(y))}{|x-y|^{N+2s}}\:dy\biggr)^2\:dx\biggr)^{1/2}
	\\
	&\le\|u\|_{L^2(B_{1-\delta})}\biggl(\int_{\R^N}\biggl(\int_{\R^N}\frac{(v(x)-v(y))^2}{|x-y|^{N+2s}}\:dy\biggr)\biggl(\int_{\R^N}\frac{(\eta(x)-\eta(y))^2}{|x-y|^{N+2s}}\:dy\biggr)\:dx\biggr)^{1/2}
	\\
	&\le C(\eta)\|u\|_{L^2(B_{1-\delta})}\biggl(\int_{\R^N}\int_{\R^N}\frac{(v(x)-v(y))^2}{|x-y|^{N+2s}}\:dy\:dx\biggr)^{1/2} = C(\eta)\|u\|_{L^2(B_{1-\delta})}\,
	[v]_{H^s(\R^N)}.
	\end{split}
	\end{equation}
	Now, in order to estimate the right hand side in \eqref{eq:estII.1}, we use that $v$ is a weak solution of  \eqref{eq:auxveryweak}; by testing against $v$ itself we obtain
\begin{align}	\label{eq:crucialestimate} 
[v]^2_{H^s(\R^N)}=&
\int_{\R^N}dx\int_{\R^N}\frac{(v(x)-v(y))^2}{|x-y|^{N+2s}}\:dy
=2\int_{\R^N}v(-\Delta)^sv\:dx=2\int_{\R^N}v D^s_{\eta_\tau,\eta}u \:dx
\\ \nonumber
=&2\int_{\R^N}v(x)dx\int_{\R^N}\frac{(\eta(x)u(x)-\eta(y)u(y))}{|x-y|^{N+2s}}\eta_\tau(|x-y|)\:dy
\\ \nonumber
=&\int_{\R^N}dx\int_{\R^N}\left(\frac{(v(x)-v(y))}{|x-y|^{\frac{N+2s}{2}}}\sqrt{\eta_\tau(|x-y|)}\right)\left(\frac{(\eta(x)u(x)-\eta(y)u(y))}{|x-y|^{\frac{N+2s}{2}}}\sqrt{\eta_\tau(|x-y|)}\right)dy
\\ \nonumber
\le&\left(\int_{\R^N}dx\int_{\R^N}\frac{(v(x)-v(y))^2}{|x-y|^{N+2s}}\eta_\tau(|x-y|)\:dy\right)^{1/2} \cdot
\\ \nonumber
&\left(\int_{\R^N}dx\int_{\R^N}\frac{(\eta(x)u(x)-\eta(y)u(y))^2}{|x-y|^{N+2s}}\eta_\tau(|x-y|)dy\right)^{1/2}
\\ \nonumber
\le&[v]_{H^s(\R^N)}\sqrt{G^s_{\eta_\tau,\eta}(u)},
\end{align}
	where in the first estimate we applied the H\"older inequality with exponent $2$ and in the second one we took into account that $\|\eta_\tau\|_{L^\infty((0,\infty))}=1$. Summarising, \begin{equation}\label{eq:sobolev}
	[v]_{H^s(\R^N)}\le \sqrt{G^s_{\eta_\tau,\eta}(u)}
\end{equation}	
and thus by \eqref{eq:estII.1} we get
	\begin{equation}\label{eq:estII.2}
	|C_3|\le C(\eta)\|u\|_{L^2(B_{1-\delta})} 
\sqrt{G^s_{\eta_\tau,\eta}(u)}	.
	\end{equation}
Let us estimate $C_4$. Since $\eta(x)=\eta(y)=0$ for $x\in B^c_{1-\delta}$ and $y\in B^c_{1-2\delta}$ we have 
	\begin{equation*}
	\begin{split}
	|C_4|&=\left|\int_{B_{1-\delta}^{c}}u(x)dx\int_{\R^N}\frac{(v(x)-v(y))(\eta(x)-\eta(y))}{|x-y|^{N+2s}}\:dy\right|
	\\
	&=\left|\int_{B_{1-\delta}^{c}}u(x)dx\int_{B_{1-2\delta}}\frac{(v(x)-v(y))(\eta(x)-\eta(y))}{|x-y|^{N+2s}}\:dy\right|,
	\end{split}
	\end{equation*}
	and then 
	\begin{equation}\label{eq:estII.3}
	\begin{split}
	|C_4|&=\left|\int_{B_{1-\delta}^{c}}\frac{u(x)}{1+|x|^{N+2s}}dx\int_{B_{1-2\delta}}\left(\frac{1+|x|^{N+2s}}{|x-y|^{N+2s}}\right)v(y)\eta(y)\:dy\right|
	\\
	&\le\int_{B_{1-\delta}^{c}}\frac{|u(x)|}{1+|x|^{N+2s}}dx\int_{B_{1-2\delta}}\left(1+\frac{1+|y|}{|x-y|}\right)^{N+2s}|v(y)|\:dy
	\\
	&\le C(\delta)\|u\|_{L_s^{1}(\R^{N})}[v]_{H^{s}(\R^N)}
	\\
	&\le C(\delta)\|u\|_{L_s^{1}(\R^{N})}\sqrt{G^s_{\eta_\tau,\eta}(u)}.
	\end{split}
	\end{equation}
In the third inequality we exploited \eqref{eq:sobolev}, the second one follows from the H\"{o}lder inequality and the fractional Sobolev inequality, see \cite[Theorem 6.5]{DPV}; taking into account that $v$ has compact support in $B_{1-\delta}$ we have:
\[
\|v\|_{L^1(B_{1-2\delta})}\le |B_{1-2\delta}|^{\frac{2_s^*-1}{2_s^*}}
\|v\|_{L^{2_s^*}(\R^N)}\le C(\delta,N,s) [v]_{H^s(\R^N)}.
\]
Now, let us estimate the term $C_1$:
	\begin{equation}\label{eq:estI.1}\begin{split}
|C_1|&=\left|\int_{B_{1-\delta}} u v(-\Delta)^s\eta\:dx\right|\le\|u\|_{L^{(2^*_s)'}(B_{1-\delta})}\|v\|_{L^{2^*_s}(\R^N)}
\|(-\Delta)^s\eta\|_{L^\infty(\R^N)} 
\\
&\le C(\delta)[v]_{H^{s}(\R^N)}\leq 
C(\delta) \sqrt{G^s_{\eta_\tau,\eta}(u)},
\end{split}	
	\end{equation}
	where $C(\delta)>0$ can be explicitly computed and in the second and third inequalities we exploited again the fractional Sobolev inequality and \eqref{eq:sobolev}, respectively.
	
	By putting together \eqref{eq:estII.1}, \eqref{eq:estII.3} and \eqref{eq:estI.1} in \eqref{eq:splitestim} and using equality \eqref{eq:estIII} we have
	\begin{equation}
	\label{eq:finalestimate}
	G^s_{\eta_\tau,\eta}(u)\le C(\left\|u\right\|_{L^1_s(\R^N)},\left\|u\right\|_{L^2(B_{1-\delta})},\delta,\eta).
	\end{equation}
	Now, we recall that $\eta_\tau$ depends on the parameter $\tau\in(0,1/2)$ but the estimate \eqref{eq:finalestimate} is uniform with respect $\tau$ because the right-hand side is independent of $\tau$. Therefore, estimate \eqref{eq:finalestimate} finally yields
	\begin{equation}
	\label{eq:conclusion}
	[\eta u]^2_{H^s(\R^N)}\le\sup_{\tau\in(0,1/2)}G^s_{\eta_\tau,\eta}(u)<\infty,
	\end{equation}
	where the first inequality holds true in view of Fatou's Lemma. Thus $\eta u\in H^{s}(\R^N)$ and then $u\in H_{\text{loc}}^{s}(B)$. 
	
	In order to complete the proof, let us show that $\eta^2 u$ is a compactly supported weak solution of 
	 $(-\Delta)^sw=f$, with $f\in L^2(\R^N)$. This implies that $\eta^2u\in H^{2s}(\R^N)$. 

	For any $\varphi\in C^\infty_c(\R^N)$ we have
	\begin{equation*}
	\begin{split}
	\frac{1}{2}\int_{\R^N}\:dx&\int_{\R^N}\frac{(\eta^2(x)u(x)-\eta^2(y)u(y))(\varphi(x)-\varphi(y))}{|x-y|^{N+2s}}\:dy\\
	&=\int_{\R^N}\: dx\int_{\R^N}\frac{\eta^2(x)u(x)(\varphi(x)-\varphi(y))}{|x-y|^{N+2s}}\:dy\\
	&=C^{-1}_{N,s}\int_{\R^N}u(x)(-\Delta)^s(\eta^2 \varphi)(x)\:dx+\int_{\R^N}\: dx\int_{\R^N}\frac{u(x)\varphi(y)(\eta^2(y)-\eta^2(x))}{|x-y|^{N+2s}}\:dy\\
	&=\int_{\R^N}\varphi(x)\:dx\int_{\R^N}\frac{u(y)(\eta^2(x)-\eta^2(y))}{|x-y|^{N+2s}}dy, 
	\end{split}
	\end{equation*}
	where in the last equality we used that $\eta^2\varphi\in C^\infty_c(B)$ and that $u$ is a very weak solution of \eqref{eq:maineq}. To conclude, we show that the function $f$ defined a.e. by
	\begin{equation}
	\label{eq:sourcef}
	f(x):=C_{N,s}\int_{\R^N}\frac{u(y)(\eta^2(x)-\eta^2(y))}{|x-y|^{N+2s}}dy
	\end{equation}
	belongs to $L^2(\R^N)$.
	We point out that
	\begin{align}
	\label{eq:explicitf}
	f(x)&=C_{N,s}\int_{\R^N}\frac{u(y)(\eta(x)+\eta(y))(\eta(x)-\eta(y))}{|x-y|^{N+2s}}dy
	\\ \nonumber
	&=C_{N,s}\int_{\R^N}\frac{\left[2u(x)\eta(x)+u(y)\eta(y)-u(x)\eta(x)+\eta(x)(u(y)-u(x))\right](\eta(x)-\eta(y))}{|x-y|^{N+2s}}dy
	\\ \nonumber
	&=2\eta(x)u(x)(-\Delta)^s\eta(x)-I_s(\eta,\eta u)(x)-\eta(x)I_s(\eta,u)(x).
	\end{align}
	The last equality shows that $f\in L^2(\R^N)$, as $I_s(\eta,\eta u)$ and $\eta I_s(\eta,u)$ belong to $L^2(\R^N)$ in view of Remark \ref{rem:stimeI} and $\eta u(-\Delta)^s\eta$, supported in $B_{1-2\delta}$, belongs to $L^2(\R^N)$ since $u\in L^2_{\rm loc}(B)$.
	
We thus obtain that $\eta^2 u\in H^{2s}(\R^N)$, and so $u\in H^{2s}_{\text{loc}}(B)$.
\end{proof}

\section{Full regularity}

In this section we prove that a very weak $s$-harmonic function $u$ is actually a classical $s$-harmonic function hence locally smooth. 
To do this, we use the fact that $u$ is locally bounded and by fractional De Giorgi estimates proved in \cite[Theorem 1.4.]{BrLiSc} also belongs to $C^{0,\gamma}_{\rm loc}(B)$ for every $\gamma\in\left(0,\min\{2s,1\}\right)$.
Then, in Propositions \ref{prop:estimatesIs} and \ref{prop:estimateetaIs} we prove that the operator $I_s$ enjoys useful H\"older continuity properties that allow us to exploit that the function $\eta^2u$ solves the equation $(-\Delta)^sw=f\in C^{0,\gamma}(\R^N)$, and then $u\in C^{\gamma+2s}_{\rm loc}(B)\cap L^1_s(\R^N)$. Therefore $u$ is $s$-harmonic in $B$ in the classical sense and, by Theorem \ref{th:analyticity}, is real analytic in $B$.
\begin{theorem}
	\label{th:BrLiSc}
Let $u\in H^{2s}_{\rm loc}(B)\cap L^1_s(\R^N)$ a very weak $s$-harmonic function in $B$. Then $u\in C^{0,\alpha}_{\rm loc}(B)$ for every $\alpha=\alpha(s)\in(0,\min\{2s,1\})$.
\end{theorem}
\begin{proof}
Since $u\in H^{2s}_{\rm loc}(B)\cap L^1_s(\R^N)$ then $u$ is a local weak solution of \eqref{eq:maineq} and by \cite[Theorem 3.2., Remark 3.3.]{BrLiSc} one has that $u\in L^\infty_{\rm loc}(B)$. Then the claim plainly follows from \cite[Theorem 1.4.]{BrLiSc}.
\end{proof}

\begin{proposition}
	\label{prop:estimatesIs}
	Let $s\in(0,1)$ and $\alpha\in(s,\min\{2s,1\})$. For any $f\in C^{0,\alpha}(\R^N)$ and $g\in C^{0,1}(\R^N)$ we have that $I_s(f,g)\in C^{0,\gamma(\alpha,s)}(\R^N)$ where 
	$\gamma(\alpha,s):=\begin{cases}
		2\alpha-2s\quad\text{if}\quad0<s\le\frac12 \\
		\alpha-2s+1\quad\text{if}\quad\frac12<s<1
	\end{cases}$ and
	$$
	[I_s(f,g)]_{C^{0,\gamma(\alpha,s)}(\R^N)}\le C[f]_{C^{0,\alpha}(\R^N)}[g]_{C^{0,1}(\R^N)}
	$$
\end{proposition}
\begin{proof}
	Let $x, x'\in\R^N$, $x\ne x'$, and let $R:=|x-x'|$. We estimate
	\begin{align}
	\label{eq:VarIs}
	I_s(f,g)(x)&-I_s(f,g)(x')
	\\ \nonumber
	&=\int_{\R^N}\frac{(f(x)-f(x-y))(g(x)-g(x-y))-(f(x')-f(x'-y))(g(x')-g(x'-y))}{|y|^{N+2s}}dy.
	\end{align}
	By adding ad subtracting $(f(x)-f(x-y))(g(x')-g(x'-y))$ in the numerator of the integrand in \eqref{eq:VarIs} we can equivalently write
	$$
	I_s(f,g)(x)-I_s(f,g)(x')=J_{1}+J_{2}+J_{3}+J_{4}
	$$
	where
	$$
	J_{1}:=\int_{B_R}\frac{(f(x)-f(x-y))\left[(g(x)-g(x-y))+(g(x'-y)-g(x'))\right]}{|y|^{N+2s}}dy,
	$$
	$$
	J_{2}:=\int_{B^c_R}\frac{(f(x)-f(x-y))\left[(g(x)-g(x'))+(g(x'-y)-g(x-y))\right]}{|y|^{N+2s}}dy,
	$$
	$$
	J_{3}:=\int_{B_R}\frac{(g(x')-g(x'-y))\left[(f(x)-f(x-y))+(f(x'-y)-f(x'))\right]}{|y|^{N+2s}}dy,
	$$
	$$
	J_{4}:=\int_{B^c_R}\frac{(g(x')-g(x'-y))\left[(f(x)-f(x'))+(f(x'-y)-f(x-y))\right]}{|y|^{N+2s}}dy.
	$$ 
Now if $s\le\frac12$ we have 
	$$
	|J_{1}|\le C[f]_{C^\alpha(\R^N)}[g]_{C^{0,1}(\R^N)}\int_{B_R}\frac{dy}{|y|^{N+2s-\alpha-1}}=C[f]_{C^\alpha(\R^N)}[g]_{C^{0,1}(\R^N)}|x-x'|^{\alpha-2s+1},
	$$
	and
	$$
	|J_{2}|\le C[f]_{C^\alpha(\R^N)}[g]_{C^{0,1}(\R^N)}|x-x'|\int_{B^c_R}\frac{dy}{|y|^{N+2s-\alpha}}=C[f]_{C^\alpha(\R^N)}[g]_{C^{0,1}(\R^N)}|x-x'|^{\alpha-2s+1}.
	$$
	The estimate of $J_{3}$ is analogous to the one of $J_{1}$ while for $J_{4}$ we have 
	$$
	|J_{4}|\le C[f]_{C^\alpha(\R^N)}[g]_{C^{\alpha}(\R^N)}|x-x'|^{\alpha}\int_{B^c_R}\frac{dy}{|y|^{N+2s-\alpha}}=C[f]_{C^\alpha(\R^N)}[g]_{C^\alpha(\R^N)}|x-x'|^{2\alpha-2s}.
	$$
	Since $2\alpha-2s<\alpha-2s+1$ we get the thesis.
	
	If $s>\frac12$ the estimates of $J_{1}, J_{2}, J_{3}$ are analogous to the previous case, while for $J_{4}$ we can write
	$$
	|J_{4}|\le C[f]_{C^\alpha(\R^N)}[g]_{C^{0,1}(\R^N)}|x-x'|^{\alpha}\int_{B^c_R}\frac{dy}{|y|^{N+2s-1}}=C[f]_{C^\alpha(\R^N)}[g]_{C^{0,1}(\R^N)}|x-x'|^{\alpha-2s+1}.
	$$
	Hence $I_s(f,g)\in C^{\alpha-2s+1}(\R^N)$.
	\end{proof}

\begin{proposition}
	\label{prop:estimateetaIs}
	Let $s\in(0,1)$, $\alpha\in(s,\min\{2s,1\})$ and $\beta=\beta(\alpha):=\frac{\alpha}{\alpha+1}<\frac{\alpha}{2s}$. If $f\in C^{0,\alpha}_{\text{loc}}(B)\cap L^1_s(\R^N)$ and $\eta\in C^\infty_c(B)$ then $\eta I_s(\eta,f)\in C^{0,\gamma}(\R^N)$ where $\gamma:=(\alpha-2s+1)\beta$ and
	$$
	\left\|\eta I_s(\eta,f)\right\|_{C^{0,\gamma}(\R^N)}\le C\left(\left\|f\right\|_{C^{0,\alpha}(B_{r'})}+\left\|f\right\|_{L^1_s(\R^N)}\right)
	$$
	where $r'\in(0,1)$ is such that $B_{r'}={\rm supp}\,\eta$.
\end{proposition}
\begin{proof}
	Let $B_r\Subset B$ with $0<r'<r$. Then $f\in C^{0,\alpha}(B_r)$. Since
	$$
	\left\|\eta I_s(\eta,f)\right\|_{C^{0,\gamma}(\R^N)}\le\left\|\eta\right\|_{C^{0,\gamma}(\R^N)}\left\|I_s(\eta,f)\right\|_{L^\infty(B_{r'})}+\left\|\eta\right\|_{L^\infty(\R^N)}\left\|I_s(\eta,f)\right\|_{C^{0,\gamma}(B_{r'})}
	$$
	we reduce to prove that $I_s(\eta,f)\in C^{0,\gamma}(B_{r'})$. By \eqref{eq:stimaIs2} we already know that $I_s(\eta,f)\in L^\infty(B_{r'})$.
	
	Now, let $x,x'\in B_{r'}$, $x\ne x'$. We set $R:=|x-x'|$. We assume that $R<\left(\frac12\right)^{\frac{1}{1-\beta}}$ otherwise the proof is done. We observe that thanks to the required upper bound on $R$ the following inclusions of sets
	$$B_{R^\beta}(x')\subset B_{\frac 32R^\beta}(x)\quad\text{and}\quad B^c_{R^\beta}(x')\subset B^c_{\frac{R^\beta}{2}}(x),$$
	hold true.
	We write
	\begin{align}
	\label{eq:stimaB}
	|I_s(\eta,f)(x)-&I_s(\eta,f)(x')|\\
	&=\left|\int_{\R^N}\frac{(f(x)-f(y))(\eta(x)-\eta(y))}{|x-y|^{N+2s}}dy-\int_{\R^N}\frac{(f(x')-f(y))(\eta(x')-\eta(y))}{|x'-y|^{N+2s}}dy\right| 
	\\ \nonumber
	\leq&\int_{B_{R^\beta}(x')}\frac{|f(x)-f(y)||\eta(x)-\eta(y)|}{|x-y|^{N+2s}}dy
	\\ \nonumber
	&+\int_{B_{R^\beta}(x')}\frac{|f(x')-f(y)||\eta(x')-\eta(y)|}{|x'-y|^{N+2s}}dy 
	\\ \nonumber
	&+\left|\int_{B^c_{R^\beta}(x')}\frac{(f(x)-f(y))(\eta(x)-\eta(y))-\frac{|x-y|^{N+2s}}{|x'-y|^{N+2s}}(f(x')-f(y))(\eta(x')-\eta(y))}{|x-y|^{N+2s}}dy\right|
	\\ \nonumber
	\leq&\int_{B_{\frac32R^\beta}(x)}\frac{|f(x)-f(y)||\eta(x)-\eta(y)|}{|x-y|^{N+2s}}dy
	\\ \nonumber
	&+\int_{B_{R^\beta}(x')}\frac{|f(x')-f(y)||\eta(x')-\eta(y)|}{|x'-y|^{N+2s}}dy 
	\\ \nonumber
	&+\left|\int_{B^c_{R^\beta}(x')}\frac{(f(x)-f(y))(\eta(x)-\eta(y))-\frac{|x-y|^{N+2s}}{|x'-y|^{N+2s}}(f(x')-f(y))(\eta(x')-\eta(y))}{|x-y|^{N+2s}}dy\right|
	\\ \nonumber
	\leq& C\left[f\right]_{C^{0,\alpha}(B_{r})}R^{(\alpha-2s+1)\beta}
	\\ \nonumber
	&+\left|\int_{B^c_{R^\beta}(x')}\frac{(f(x)-f(y))(\eta(x)-\eta(y))-\frac{|x-y|^{N+2s}}{|x'-y|^{N+2s}}(f(x')-f(y))(\eta(x')-\eta(y))}{|x-y|^{N+2s}}dy\right|
	\\ \nonumber
	=&:C\left[f\right]_{C^{0,\alpha}(B_{r})}R^{(\alpha-2s+1)\beta}+|B|.
	\end{align}
	By adding and subtracting $\frac{(f(x')-f(y))(\eta(x')-\eta(y))}{|x-y|^{N+2s}}$ to the integrand defining $B$ we have 
	\begin{align*}
	|B|\leq&\left|\int_{B^c_{R^\beta}(x')}\frac{(f(x)-f(y))(\eta(x)-\eta(y))-(f(x')-f(y))(\eta(x')-\eta(y))}{|x-y|^{N+2s}}dy\right| 
	\\
	&+\int_{B^c_{R^\beta}(x')}\left|\frac{1}{|x-y|^{N+2s}}-\frac{1}{|x'-y|^{N+2s}}\right||f(x')-f(y)||\eta(x')-\eta(y)|dy
	\\
	\leq&\int_{B^c_{R^\beta}(x')}\frac{|f(x)\eta(x)-f(x')\eta(x')|}{|x-y|^{N+2s}}dy+\int_{B^c_{R^\beta}(x')}\frac{|f(y)||\eta(x)-\eta(x')|}{|x-y|^{N+2s}}dy
	\\
	&+\int_{B^c_{R^\beta}(x')}\frac{|\eta(y)||f(x)-f(x')|}{|x-y|^{N+2s}}dy
	\\
	&+\int_{B^c_{R^\beta}(x')}\left|\frac{1}{|x-y|^{N+2s}}-\frac{1}{|x'-y|^{N+2s}}\right||f(x')-f(y)||\eta(x')-\eta(y)|dy
	\\
	\leq&\int_{B^c_{\frac{R^\beta}{2}}(x)}\frac{|f(x)\eta(x)-f(x')\eta(x')|}{|x-y|^{N+2s}}dy+\int_{B^c_{\frac{R^\beta}{2}}(x)}\frac{|f(y)||\eta(x)-\eta(x')|}{|x-y|^{N+2s}}dy
	\\
	&+\int_{B^c_{\frac{R^\beta}{2}}(x)}\frac{|\eta(y)||f(x)-f(x')|}{|x-y|^{N+2s}}dy
	\\
	&+\int_{B^c_{R^\beta}(x')}\left|\frac{1}{|x-y|^{N+2s}}-\frac{1}{|x'-y|^{N+2s}}\right||f(x')-f(y)||\eta(x')-\eta(y)|dy
	\\
	=&:B_1+B_2+B_3+B_4.
	\end{align*}
%	where the last inequality follows from the fact that $B^c_{R^\beta}(x')\subset B^c_{R^\beta/2}(x)$ since $R<\left(\frac12\right)^{\frac{1}{1-\beta}}$.
	Now we estimate $B_i$ for any $i\in\{1,2,3,4\}$.
	\begin{equation}
	\label{eq:stimaB1}
	B_1\le C[\eta u]_{C^{0,\alpha}(\R^N)}R^\alpha\int_{B^c_{R^\beta/2}(x)}\frac{dy}{|x-y|^{N+2s}}\le C[\eta u]_{C^{0,\alpha}(\R^N)}R^{\alpha-2\beta s}.
	\end{equation}
	\begin{equation}
	\label{eq:stimaB2}
	\begin{split}
	B_2&\le CR\left(\int_{B^c_{R^\beta/2}(x)\cap B_r}\frac{|f(y)|}{|x-y|^{N+2s}}dy+\int_{B^c_{R^\beta/2}(x)\cap B^c_r}\frac{|f(y)|}{|x-y|^{N+2s}}dy\right)\\
	&\le CR\left(\left\|f\right\|_{L^\infty(B_r)}\int_{B^c_{R^\beta/2}(x)}\frac{dy}{|x-y|^{N+2s}}+C(r,r')\int_{B^c_r}\frac{|f(y)|}{1+|y|^{N+2s}}dy\right)\\
	&\le CR\left(C_1\left\|f\right\|_{L^\infty(B_r)}R^{-2\beta s}+C(r,r')\left\|f\right\|_{L^1_s(\R^N)}\right)\\
	&\le C(r,r',N,s)\left(\left\|f\right\|_{L^\infty(B_r)}+\left\|f\right\|_{L^1_s(\R^N)}\right)R^{1-2\beta s}\\
	&\le C(r,r',N,s)\left(\left\|f\right\|_{L^\infty(B_r)}+\left\|f\right\|_{L^1_s(\R^N)}\right)R^{\alpha-2\beta s}.
	\end{split}
	\end{equation}
	\begin{equation}
	\label{eq:stimaB3}
	B_3\le C\left[f\right]_{C^{0,\alpha}(B_{r'})}R^\alpha\int_{B^c_{R^\beta/2}(x)}\frac{dy}{|x-y|^{N+2s}}\le C\left[f\right]_{C^{0,\alpha}(B_{r'})}R^{\alpha-2\beta s}.
	\end{equation}
	For $B_4$ we write
	\begin{align*}
	B_4=&\int_{B^c_{R^\beta}(x')\cap B_r}\left|\frac{1}{|x-y|^{N+2s}}-\frac{1}{|x'-y|^{N+2s}}\right||f(x')-f(y)||\eta(x')-\eta(y)|dy 
	\\
	&+\int_{B^c_{R^\beta}(x')\cap B^c_r}\left|\frac{1}{|x-y|^{N+2s}}-\frac{1}{|x'-y|^{N+2s}}\right||f(x')-f(y)||\eta(x')-\eta(y)|dy 
	\\
	=&:D_{1}+D_{2}.
	\end{align*}
	Using the fundamental Theorem of Calculus we write
	$$
	\left|\frac{1}{|x-y|^{N+2s}}-\frac{1}{|x'-y|^{N+2s}}\right|\le C_{N,s}|x-x'||x'-y|^{-N-2s-1}.
	$$
	Therefore
	\begin{equation}
	\label{eq:stimaD1}
	\begin{split}
	D_{1}&\le C\left[f\right]_{C^{0,\alpha}(B_{r'})}R\int_{B^c_{R^\beta}(x')}|x'-y|^{\alpha+1}|x'-y|^{-N-2s-1}dy\\
	&=C\left[f\right]_{C^{0,\alpha}(B_{r'})}R^{1+\beta(\alpha-2s)}\\
	&\le C\left[f\right]_{C^{0,\alpha}(B_{r'})}R^{(\alpha-2s+1)\beta}.
	\end{split}
	\end{equation}
	To conclude for $D_{2}$ we have
	\begin{equation}
	\label{eq:stimaD2}
	\begin{split}
	D_{2}&\le CR\int_{B^c_{R^\beta}(x')\cap B^c_r}|f(x')-f(y)||x'-y||x'-y|^{-N-2s-1}dy\\
	&\le CR\left(\left\|f\right\|_{L^\infty(B_{r'})}\int_{B_{R^\beta}(x')}\frac{dy}{|x'-y|^{N+2s}}+\int_{B^c_r}\frac{|f(y)|}{|x'-y|^{N+2s}}dy\right)\\
	&\le CR\left(C_1\left\|f\right\|_{L^\infty(B_{r'})}R^{-2\beta s}+C(r,r')\int_{\R^N}\frac{|f(y)|}{1+|y|^{N+2s}}dy\right)\\
	&\le C(r,r',N,s)\left(\left\|f\right\|_{L^\infty(B_{r'})}+\left\|f\right\|_{L^1_s(\R^N)}\right)R^{1-2\beta s}\\
	&\le C(r,r',N,s)\left(\left\|f\right\|_{L^\infty(B_{r'})}+\left\|f\right\|_{L^1_s(\R^N)}\right)R^{\alpha-2\beta s}.
	\end{split}
	\end{equation}
	Putting \eqref{eq:stimaB1}, \eqref{eq:stimaB2}, \eqref{eq:stimaB3}, \eqref{eq:stimaD1},\eqref{eq:stimaD2} into \eqref{eq:stimaB} and taking into account that $(\alpha-2s+1)\beta=\alpha-2\beta s$ by the choice of $\beta$, our claim is proved.
	\end{proof}
\begin{theorem}
	\label{th:pointwiseholder}
	Let $u\in C^{0,\gamma}_{\text{loc}}(B)\cap L^1_s(\R^N)$ for some $\gamma\in(0,1)$ a very weak $s$-harmonic function in $B$. Then $u\in C^{\gamma+2s}_{\text{loc}}(B)\cap L^1_s(\R^N)$ hence $u$ is $s$-harmonic in the classical sense in $B$.
\end{theorem}
\begin{proof}
	Let $\eta\in C^\infty_c(B)$. By Theorem \ref{th:sobolevregularity} the function $\eta^2 u$ is a weak solution of $(-\Delta)^sw=f$ in $\R^N$ with $f:=2\eta u(-\Delta)^s\eta-I_s(\eta,\eta u)-\eta I_s(\eta,u)$. Moreover by applying Theorem \ref{th:BrLiSc} with $\alpha\in(s,\min\{2s,1\})$ and Propositions \ref{prop:estimatesIs}, \ref{prop:estimateetaIs} we have that $f\in C^{0,\gamma}(\R^N)$. Therefore by using Schauder estimates for bounded weak solutions to $(-\Delta)^sw=f$ (see \cite[Proposition 2.8]{sil} or \cite[Theorem 15]{stinga}) it follows that $\eta^2 u\in C^{0,\gamma+2s}(\R^N)$ if $0<\gamma+2s<1$ or $\eta^2 u\in C^{1,\gamma+2s-1}(\R^N)$ if $1<\gamma+2s<2$ or  $\eta^2 u\in C^{2,\gamma+2s-2}(\R^N)$ if $2<\gamma+2s<3$. If $\gamma+2s\in\N$ apply Propositions \ref{prop:estimatesIs}, \ref{prop:estimateetaIs} replacing $\alpha$ with $\alpha_1$ such that $\alpha_1+2s\notin\N$.
	By the arbitrariness of $\eta$ we get the thesis.
	\end{proof}
	
\begin{remark}
	Notice that when $s\in\left(0,\frac 12\right)$ Theorem \ref{th:pointwiseholder} also follows from \cite[Theorem 3.5]{BDLMbS} applied with $p=2$.
\end{remark}
\begin{theorem}
	\label{th:analyticity}
	Let $u\in C^{\gamma+2s}_{\text{loc}}(B)\cap L^1_s(\R^N)$ for some $\gamma\in(0,1)$ a very weak $s$-harmonic function in $B$. Then $u$ is real analytic in $B$.
\end{theorem}
\begin{proof}
	Let $\delta\in(0,1)$ and $R=R(\delta):=1-\frac\delta4$, $r=r(\delta):=1-\frac\delta2$, $r_0=r_0(\delta):=1-\delta$. By Theorem \ref{th:pointwiseholder} $u$ is a classical solution of
	\begin{equation}
	\label{eq:pointwisesharmonic}
		\begin{cases}
	(-\Delta)^sw=0\quad\text{in}\quad B_{r} \\
	w=h\quad\text{in}\quad B^c_{r},
	\end{cases}
	\end{equation}
	where
		\begin{equation*}
	h:=\begin{cases}
	u\quad\text{in}\quad B \\
	g\quad\text{in}\quad B^c.
	\end{cases}
	\end{equation*}
	Since $h\in C(\overline{B_R})\cap L^1_s(\R^N)$ if for $\rho>0$ we set $$P_\rho(x,y):=C_{N,s}\left(\frac{\rho^2-|x|^2}{|y|^2-\rho^2}\right)^s\frac{1}{|x-y|^N}$$ the fractional Poisson kernel (see \cite{bucur}) for $B_\rho$ for any $x\in B_\rho$ and $y\in B_\rho^c$, the function 
	\begin{equation}
	\label{eq:poissonkernelsolution}
	u_h(x):=\int_{B_{r}^c}P_{r}(x,y)h(y)dy
	\end{equation} is well posed for every $x\in B_{r}$. Indeed
	\begin{equation}
	\label{eq:estimatepoissonkernelsolution}
	\begin{split}
	|u_h(x)|&\le\int_{B_R\setminus B_{r}}P_{r}(x,y)|h(y)|dy+\int_{B\setminus B_R}P_{r}(x,y)|h(y)|dy+\int_{B^c}P_{r}(x,y)|h(y)|dy \\
	&=\int_{B_R\setminus B_{r}}P_{r}(x,y)|u(y)|dy+\int_{B\setminus B_R}P_{r}(x,y)|u(y)|dy+\int_{B^c}P_{r}(x,y)|g(y)|dy \\
	&\le\left\|u\right\|_{L^\infty(B_R)}\int_{B_{r}^c}P_{r}(x,y)dy+\frac{r^{2s}}{(R^2-r^2)^s(R-r)^N}\left\|u\right\|_{L^1(B)}+C_{r,R,N,s}\int_{\R^N}\frac{|u(y)|}{1+|y|^{N+2s}}dy\\
	&\le C(r,R,N,s)\left(\left\|u\right\|_{L^\infty(B_R)}+\left\|u\right\|_{L^1_s(\R^N)}\right).
	\end{split}
	\end{equation}
	By \cite[Theorem 2.10.]{bucur}, the function $u_h$ is a classical solution of \eqref{eq:pointwisesharmonic}. By uniqueness of solutions of \eqref{eq:pointwisesharmonic} (see \cite[Theorem 3.3.2.]{bucval}) we conclude that $u_h=u$ in $B_{r}$. 
	
%	In particular, the fact that $u$ is $s$-harmonic in the classical sense in $B_{r}$ is characterized by the equality
%	\begin{equation}
%	\label{eq:convsmeankernel}
%	u(x)=(u\ast A_\varrho)(x)
%	\end{equation} 
%	for every $x\in B_{r}$ and $0<\varrho<r-|x|$, where the kernel $A_\varrho$ is the $s$-mean kernel defined as
%	$$A_\varrho(z):=\begin{cases*}
%	P_\varrho(0,z)\quad\text{if}\quad|z|>\varrho \\
%	0\quad\text{if}\quad|z|\le \varrho.
%	\end{cases*}
%	$$
%	In particular from \eqref{eq:convsmeankernel} we deduce that
%	$$
%	D^\iota u(x)=(u\ast D^\iota A_\varrho)(x)
%	$$
%	for every $\iota\in\N^N_0$ and every $x\in B_{r_0}$. 

Moreover, for every $y\in B_r^c$ the function
$$
B_{r_0}\ni x\mapsto\frac{(r^2-|x|^2)^s}{|x-y|^N}
$$
is smooth, and it's easy to check that
\begin{equation}
\label{eq:estimatespoissonkernel}
|(\partial_x^\iota P_r)(x,y)|\le c^{|\iota|}\iota! C(r,r_0,N,s)P_r(x,y)
\end{equation}
for every $x\in B_{r_0}$, $y\in B^c_r$ and $\iota\in\N_0^N$.
	Therefore by differentiating under integral sign formula \eqref{eq:poissonkernelsolution} by estimates \eqref{eq:estimatepoissonkernelsolution}, \eqref{eq:estimatespoissonkernel} we have 
	$$
	\left\|D^\iota u\right\|_{L^\infty(B_{r_0})}\le c^{|\iota|}\iota!C(R,r,r_0,N,s)\left(\left\|u\right\|_{L^\infty(B_R)}+\left\|u\right\|_{L^1_s(\R^N)}\right)
	$$
	for any $\iota\in\N^N_0$.
	From the arbitrariness of $\delta\in(0,1)$ we get the thesis.
	\end{proof}

	\end{document}